\documentclass{amsart}
\usepackage{mathrsfs}
\usepackage{graphicx}
\usepackage{amssymb}
\usepackage{amsmath}
\usepackage{amsthm}
\usepackage{galois}
\usepackage{float}
\usepackage{array}

\allowdisplaybreaks
\restylefloat{figure}
\newtheorem{thm}{Theorem}[section]
\newtheorem{cor}[thm]{Corollary}
\newtheorem{lem}[thm]{Lemma}
\newtheorem{prop}[thm]{Proposition}
\theoremstyle{definition}
\newtheorem{defn}[thm]{Definition}
\newtheorem{definition}[thm]{Definition}

\newtheorem{example}[thm]{Example}
\newtheorem{ques}[thm]{Question}

\theoremstyle{remark}

\numberwithin{equation}{section}

\newcommand{\ind}{\mathrm{ind}}

\newcommand{\fpc}{\mathrm{Fpc}}
\newcommand{\F}{\mathbf{F}}

\newcommand{\out}{\mathrm{Out}}
\newcommand{\B}{\mathcal{B}}

\begin{document}
\title{On the bounded index property\\ for products of aspherical polyhedra}
\author{Qiang ZHANG and Shengkui YE}
\address{School of Mathematics and Statistics, Xi'an Jiaotong University,
Xi'an 710049, China}
\email{zhangq.math@mail.xjtu.edu.cn}

\address{Department of Mathematical Sciences, Xi'an Jiaotong-Liverpool University, Jiangsu, China}
\email{shengkui.ye@xjtlu.edu.cn}

\thanks{The authors are partially supported by NSFC Grants \#11771345, \#11961131004 and \#11971389.}

\begin{abstract}
A compact polyhedron $X$ is said to have the
Bounded Index Property for Homotopy Equivalences (BIPHE) if there is a finite bound $\B$ such that for
any homotopy equivalence $f:X\rightarrow X$ and
any fixed point class $\mathbf{F}$ of $f$, the index $|\mathrm{ind}(f,\mathbf{F})|\leq \mathcal{B}$.
In this note, we consider the product of compact polyhedra, and
give some sufficient conditions for it to have BIPHE. Moreover, we show that products of closed Riemannian manifolds with
negative sectional curvature, in particular hyperbolic manifolds, have BIPHE, which gives an affirmative answer to a special
case of a question asked by Boju Jiang.
\end{abstract}

\maketitle





\section{Introduction}

Fixed point theory studies fixed points of a self-map $f$ of a space $X$.
Nielsen fixed point theory, in particular, is concerned with the properties
of the fixed point set
\begin{equation*}
\mathrm{Fix} f:=\{x\in X|f(x)=x\}
\end{equation*}
that are invariant under homotopy of the map $f$ (see \cite{fp1} for an
introduction).

The fixed point set $\mathrm{Fix} f$ splits into a disjoint union of \emph{%
fixed point classes}: two fixed points $a$ and $a^{\prime }$ are in the same
class if and only if there is a lifting $\tilde f: \widetilde X\to
\widetilde X$ of $f$ such that $a, a^{\prime }\in p(\mathrm{Fix} \tilde f)$,
where $p:\widetilde X\to X$ is the universal cover. Let $\mathrm{Fpc}(f)$
denote the set of all the fixed point classes of $f$. For each fixed point
class $\mathbf{F}\in \mathrm{Fpc}(f)$, a homotopy invariant \emph{index} $%
\mathrm{ind}(f,\mathbf{F})\in \mathbb{Z}$ is well-defined. A fixed point
class is \emph{essential} if its index is non-zero. The number of essential
fixed point classes of $f$ is called the $Nielsen$ $number$ of $f$, denoted
by $N(f)$. The famous Lefschetz-Hopf theorem says that the sum of the
indices of the fixed points of $f$ is equal to the $Lefschetz$ $number$ $%
L(f) $, which is defined as
\begin{equation*}
L(f):=\sum_q(-1)^q\mathrm{Trace} (f_*: H_q(X;\mathbb{Q})\to H_q(X;\mathbb{Q}%
)).
\end{equation*}

In this note, all maps considered are continuous, and all spaces are triangulable, namely, they are
homeomorphic to polyhedra. A compact polyhedron $X$ is said to have the
\emph{Bounded Index Property (BIP)}(resp. \emph{Bounded Index Property for
Homeomorphisms (BIPH)}, \emph{Bounded Index Property for Homotopy
Equivalences (BIPHE)}) if there is an integer $\mathcal{B}>0$ such that for
any map (resp. homeomorphism, homotopy equivalence) $f:X\rightarrow X$ and
any fixed point class $\mathbf{F}$ of $f$, the index $|\mathrm{ind}(f,%
\mathbf{F})|\leq \mathcal{B}$. Clearly, if $X$ has BIP, then $X$ has BIPHE
and hence has BIPH. For an aspherical closed manifold $M$, if the well-known
Borel conjecture (any homotopy equivalence $f:M\rightarrow M$ is homotopic
to a homeomorphism $g:M\rightarrow M$) is true, then $M$ has BIPHE if and
only if it has BIPH.

In \cite{JG}, Jiang and Guo proved that compact surfaces with negative Euler
characteristics have BIPH. Later, Jiang \cite{fp2} showed that graphs and
surfaces with negative Euler characteristics not only have BIPH but also
have BIP (see \cite{K1}, \cite{K2} and \cite{JWZ} for some parallel
results). Moreover, Jiang asked the following question:

\begin{ques}
(\cite[Qusetion 3]{fp2}) Does every compact aspherical polyhedron $X$ (i.e. $%
\pi_i(X)=0$ for all $i>1$) have BIP or BIPH?
\end{ques}

In \cite{Mc}, McCord showed that infrasolvmanifolds (manifolds which admit a
finite cover by a compact solvmanifold) have BIP. In \cite{JW}, Jiang and
Wang showed that geometric 3-manifolds have BIPH for orientation-preserving
self-homeomorphisms: the index of each essential fixed point class is $\pm 1$%
. In \cite{Z}, the first author showed that orientable compact Seifert
3-manifolds with hyperbolic orbifolds have BIPH, and later in \cite{Z2, Z3},
he showed that compact hyperbolic $n$-manifolds (not necessarily orientable)
also have BIPH. Recently, in \cite{ZZ}, Zhang and Zhao showed that
products of hyperbolic surfaces have BIPH.

Note that in \cite[Section 6]{fp2}, Jiang gave an example that showed that BIPH is not preserved by taking products: the 3-sphere $S^3$ has BIPH while the product $S^3\times S^3$ does not have BIPH. In this note, we consider the product of connected compact polyhedra, and
give some sufficient conditions for it to have BIPHE (and hence has BIPH).
The main result of this note is the following:

\begin{thm}
\label{main thm1} Suppose $X_1,\ldots,X_n$ are connected compact aspherical
polyhedra satisfying the following two conditions:

$\mathrm{(1)}$ $\pi_1(X_i)\not \cong\pi_1(X_j)$ for $i\neq j$, and all of
them are centerless and indecomposable;

$\mathrm{(2)}$ all of $X_1,\ldots,X_n$ have BIPHE.

\noindent Then the product $X_1\times\cdots\times X_n$ also has BIPHE (and
hence has BIPH).
\end{thm}

Moreover, we show that products of closed Riemannian manifolds with
negative sectional curvature have BIPHE:

\begin{thm}
\label{main thm2} Let $M=M_1\times\cdots\times M_n$ be a product of
finitely many connected closed Riemannian manifolds, each with negative
sectional curvature everywhere but not necessarily with the same dimensions
(in particular hyperbolic manifolds). Then $M$ has BIPHE.
\end{thm}

Recall that a closed $2$-dimensional Riemannian manifold $M$ with negative
sectional curvature is a closed
hyperbolic surface, and hyperbolic surfaces have BIP. As a corollary of
Theorem \ref{main thm2}, we have

\begin{cor}
A closed Riemannian manifold with negative sectional curvature everywhere
has BIPHE.
\end{cor}

To prove the above theorems, we first study automorphisms of products of groups in Section \ref{sect 2},
and give some facts about the bounded index property of fixed points in Section \ref{sect 3}. Then in Section \ref{sect 4}, we generalize the results of alternating
homeomorphisms (see \cite[Section 3]{ZZ}) to that of cyclic
homeomorphisms of products of surfaces. Finally in Section \ref{sect 5}, we show that every homotopy equivalence of products of aspherical manifolds can be homotoped to two nice forms, and taking advantage of that, we finish the proofs.


\section{Automorphisms of products of groups}\label{sect 2}

In this section, we give some facts about automorphisms of direct products of
finitely many groups.

\begin{defn}
A group $G$ is called $\emph{unfactorizable}$ if whenever $G=HK$ for
subgroups $H,K$ satisfying $hk=kh$ for any $h\in H,k\in K$ we have either $%
H=1$ or $K=1.$ If $G=H\times K$ for some groups $H,K\ $implies either $H$ or
$K$ is trivial, we call $G$ is \emph{indecomposable}.
\end{defn}

\begin{lem}
Let $G$ be an unfactorizable group and $G=\Pi_{i=1}^n A_{i}$ a direct
product. Then $A_{i}=G $ for some $i$ and all $A_{j}=1$ for $i\neq j.$
\end{lem}

\begin{proof}
Suppose that there are at least two non-trivial components $A_{1},A_{2}.$
Then $G=A_{1}A_{2}\times \Pi _{i\neq 1,2}A_{i}.$ Then $\Pi _{i\neq
1,2}A_{i}=1,$ since $G$ is unfactorizable. Therefore, $G=A_{1}A_{2}$ a
contradiction.
\end{proof}

\begin{lem}
A group $G$ is unfactorizable if and only if $G$ is centerless and
indecomposable.
\end{lem}

\begin{proof}
Suppose that $G$ is unfactorizable. Since $G=GC$ for the center subgroup $C,$
we see that $C=1.$ An unfactorizable group is obviously indecomposable.
Conversely, assume that $G$ is centerless and indecomposable. If $G=KH$ for
commuting $K,H.$ Then any element $g\in K\cap H$ is central and thus
trivial. This implies that $G=K\times H.$ Therefore, either $K$ or $H$ is
trivial.
\end{proof}

For a direct product $G=G_1\times \cdots \times G_n$, we collect together
the coordinates corresponding to isomorphic $G_i$'s and present it in the
form $G=G_1^{\,n_1}\times \cdots \times G_m^{\,n_m}$, where $n_i\geq 1$ and $%
G_i\ncong G_j$ for $1\leq i\neq j\leq m$. Given automorphisms $\phi_i\in
\mathrm{Aut}(G_i)$ for $i=1,\ldots ,n$, let $\prod_{i=1}^n \phi_i
=\phi_1\times \cdots \times \phi_n \colon G\to G$, $(g_1,\ldots ,g_n)\mapsto
(\phi_1(g_1), \ldots ,\phi_n(g_n))$ be the product of $\phi_i$'s. We have an
analogous result of \cite[Proposition 4.4]{ZVW} as follows.

\begin{prop}
\label{aut of prouct group} Let each group $G_i, i=1,\ldots,m $, be
unfactorizable, and $G=G_1^{\,n_1}\times \cdots \times G_m^{\,n_m}$ a direct
product, where $m\geq 1$, $n_i\geq 1$, and $G_i\ncong G_j$ for $i\neq j$.
Then for every $\phi \in \mathrm{Aut}(G)$, there exist automorphisms $%
\phi_{i,j}\in \mathrm{Aut}(G_i)$ and permutations $\sigma_i \in S_{n_i}$,
such that
\begin{equation*}
\phi =\sigma_1 \comp\cdots \comp \sigma_m \comp (\prod_{i=1}^m
\prod_{j=1}^{n_i} \phi_{i,j}) = \prod_{i=1}^m (\sigma_i \comp %
\prod_{j=1}^{n_i} \phi_{i,j}).
\end{equation*}
\end{prop}

\begin{proof}
For brevity, we first assume that $G=G_1\times G_2$ (here $G_1$ may be
isomorphic to $G_2$), and $\phi: G_1\times G_2\to G_1\times G_2$ an
automorphism. Let
\begin{equation*}
A_1=\{a_{1}|\phi(g_{1},1)=(a_{1},b_{1}),g_1\in
G_1\},~~B_1=\{b_{1}|\phi(g_{1},1)=(a_{1},b_{1}),g_1\in G_1\},
\end{equation*}
and
\begin{equation*}
A_2=\{a_{2}|\phi(1,g_{2} )=(a_{2},b_{2}),g_2\in
G_2\},~~B_2=\{b_{2}|\phi(1,g_{2} )=(a_{2},b_{2}),g_2\in G_2\}.
\end{equation*}
Then $G_1=A_1A_2$, $a_{1}a_{2}=a_{2}a_{1}$ for every $a_i\in A_i$; and $%
G_2=B_1B_2$, $b_{1}b_{2}=b_{2}b_{1}$ for every $b_i\in B_i$. Since $G_{1},
G_2$ are unfactorizable, we see that either $A_1$ or $A_2$ is trivial, and
either $B_1$ or $B_2$ is trivial.

If $A_1=G_1,A_2=1$, then $B_1=1, B_2=G_2$, and we have
\begin{equation*}
\phi=\phi_1\times \phi_2: G_1\times G_2\to G_1\times G_2,\quad
(g_1,g_2)\mapsto (\phi_1(g_1),\phi_2(g_2))=(a_1,b_2)
\end{equation*}
where $\phi_i\in \mathrm{Aut}(G_i)$.

If $A_1=1, A_2=G_1$, then $B_1=G_2, B_2=1$, then we have $G_1\cong G_2$, and
\begin{equation*}
\phi=\sigma\comp (\phi_1\times \phi_2): G_1\times G_2\to G_1\times G_2,\quad
(g_1,g_2)\mapsto (\phi_2(g_2),\phi_1(g_1))=(a_2, b_1)
\end{equation*}
for $\phi_i\in \mathrm{Aut}(G_i)$ and $\sigma\in S_2$ a permutation.

Now we have proved that Proposition \ref{aut of prouct group} holds for the
case $G=G_1\times G_2$. For the general case that $G=\prod_i G_i$ has more
than 2 factors, the same argument as above shows that Proposition \ref{aut
of prouct group} also holds.
\end{proof}


Since the inner automorphism group $\mathrm{Inn}(\prod_i G_i)$ of a product
is isomorphic to the product $\prod_i \mathrm{Inn}(G_i)$, summarizing the
above results, we have proved the following:

\begin{thm}
\label{aut and out of product groups} Let $G_1,\ldots, G_m $ be finitely
many centerless indecomposable groups, and $G_i\ncong G_j$ for $i\neq j$.
Then the automorphism group of the product $\prod_{i=1}^m G_i^{n_i}$
\begin{equation*}
\mathrm{Aut}(\prod_{i=1}^m G_i^{n_i})\cong \prod_{i=1}^m\bigg(\prod_{n_i}%
\mathrm{Aut}(G_{i})\bigg)\rtimes S_{n_i},
\end{equation*}
and the outer automorphism group
\begin{equation*}
\mathrm{Out}(\prod_{i=1}^m G_i^{n_i})\cong \prod_{i=1}^m\bigg(\prod_{n_i}%
\mathrm{Out}(G_{i})\bigg)\rtimes S_{n_i},
\end{equation*}
where the symmetric group $S_{n_i}$ of $n_i\geq 1$ elements acts on $%
\prod_{n_i}\mathrm{Aut}(G_{i})$ and $\prod_{n_i}\mathrm{Out}(G_{i})$ by
natural permutations.
\end{thm}

There are many examples of centerless indecomposable groups.

\begin{example}
Non-abelian free groups are centerless indecomposable groups.
\end{example}

\begin{example}\label{centerless of negatived mfd}
Non-abelian torsion-free Gromov hyperbolic groups are centerless indecomposable groups.
In particular, the fundamental group of a closed Riemannian manifold with negative
sectional curvature everywhere is centerless and indecomposable.
\end{example}

\begin{proof}
Let $G$ be a non-abelian torsion-free Gromov hyperbolic group and $1\neq \gamma \in G.$
It is well-known that the centralizer subgroup $C(\gamma )=\{g\in G\mid
g\gamma =\gamma g\}$ contains $\langle \gamma \rangle $ as a finite-index
subgroup (see \cite{bh}, Corollary 3.10, p.462). Therefore, the center of $G$
is trivial and $G$ is indecomposable.
\end{proof}


\section{Facts about the bounded index properties}\label{sect 3}

In this section, we give some facts about BIP, BIPHE and BIPH. In order to
state results conveniently, we will use the following definition.

\begin{definition}
Let $X$ be a compact polyhedron and $C$ be a family of self-maps
of $X.$ We call that $X$ has the \emph{Bounded Index Property with
respect to} $C$ (denoted by $\mathrm{BIPC}$) if there exists an integer $\B>0$
(depending only on $C$) such that for any map $f\in C$ and any fixed point
class $\mathbf{F}$ of $f$, the index $|\mathrm{ind}(f,\mathbf{F})|\leq \B$.
The minimum such a $\B$ is called the bounded index for $C$.
\end{definition}

When $C$ is the set of all self-maps (self-homeomorphisms, self-homotopy equivalences, respectively), we simply denote the $\mathrm{BIPC}$
by BIP (BIPH, BIPHE, respectively). It is obvious that $\mathrm{BIPC}_{1}$
implies $\mathrm{BIPC}_{2}$ when $C_{2}$ is a subset of $C_{1}.$

For maps of polyhedra, Jiang gave the following definition of mutant, and showed that the Nielsen fixed
point invariants are invariants of mutants (see \cite[Sect. 1]{fp2}).

\begin{definition}
Let $f:X\rightarrow X$ and $g:Y\rightarrow Y$ be self-maps of compact
connected polyhedra. We say $g$ is obtained from $f$ by \emph{commutation}, if there exist maps $\phi :X\rightarrow Y$ and $\psi
:Y\rightarrow X$ such that $f=\psi \comp \phi $  and $g=\phi \comp
\psi $. 
Say $g$ is a \emph{mutant} of $f$, if there is a finite sequence $\{f_i : X_i\to X_i | i = 1,2,\cdots,k\}$ of self-maps of compact polyhedra such that  $f=f_1$, $g=f_k$, and for each $i$, either
$X_{i+1}=X_i$ and $f_{i+1}\simeq f_i$, or
$f_{i+1}$ is obtained from $f_i$ by commutation.
\end{definition}

\begin{lem}[Jiang, \cite{fp2}]\label{mutant invar.}
Mutants have the same set of indices of essential fixed point classes,
hence also the same Lefschetz number and Nielsen number.
\end{lem}

Note that mutants give an equivalence relation on self-maps of compact polyhedra. In other words, two
self-maps $f:X\to X$ and $g:Y\to Y$ are mutant-equivalent if there exist finitely many maps of compact polyhedra $X_1=X,X_2,\ldots, X_k=Y,$
$$u_{i}:X_i\rightarrow X_{i+1},\quad v_{i}:X_{i+1}\rightarrow X_i,\quad\quad i=1,2,\cdots ,k-1,$$
such that $f\simeq v_{1}\comp u_{1}: X_1\to X_1, ~~g\simeq u_{k-1}\comp v_{k-1}:X_k\to X_k,$ and for
$i<k-1$,
$$u_{i}\comp v_{i}\simeq v_{i+1}\comp u_{i+1}.$$

For a self-map $f:X\to X$  of a connected compact polyhedron $X$, let $[f]_m$ denote the mutant-equivalent class of $f$, and for a family $C$ of self-maps of $X$, let $[C]_m:=\{[f]_m|f\in C\}$ be the set of
mutant-equivalent classes of $C$. Note that $f$
has only finitely many non-empty fixed point classes, and each is a compact
subset of $X$, we have a finite bound $\mathcal{B}_{f}$ of $\mathrm{ind}(f,%
\mathbf{F})$ for all the fixed point classes $\mathbf{F}$ of $f$.
As an immediate consequence of Lemma \ref{mutant invar.}, we have the following:

\begin{prop}\label{ind for muatant}
\label{BIP of homotopy type} For any family $C$ of self-maps of a connected compact polyhedron $X$, the set $\{\ind(f,\F)|f\in C, \F\in \fpc(f)\}$ defined on $C$ factors through the equivalence classes $[C]_{m}$. Namely, for any set $[C]_m$ of mutant-equivalent classes, we have a set of indices
$$\ind([C]_m):=\{\ind(f,\F)|f\in C, \F\in \fpc(f)\}$$
depending only on $[C]_m$. Moreover, if $[C]_m$ is finite, then $X$ has BIPC.
\end{prop}

Suppose that $X$ is aspherical and let $[X]$ (resp. $[X]_{h.e}$) be the set
of all homotopy classes of self-maps (resp. self-homotopy equivalences) of $X$%
. It is well-known that there is a bijective correspondence
\begin{equation*}
\eta:[X]\longleftrightarrow \mathrm{End}(\pi _{1}(X))/\mathrm{Inn}(\pi _{1}(X)),
\end{equation*}%
given by sending a self-map $f$ to the induced endomorphism $f_{\pi}$ of the fundamental group, where the inner
automorphism group $\mathrm{Inn}(\pi _{1}(X))$ acts on the semigroup $%
\mathrm{End}(\pi _{1}(X))$ of all the endomorphisms of $\pi _{1}(X)$ by
composition. Note that $\eta$ induces a bijection (still denoted by $\eta$)
\begin{equation*}
\eta: \lbrack X]_{h.e}\longleftrightarrow \mathrm{Out}(\pi _{1}(X)).
\end{equation*}%

For any self-map $f$ of $X$, let $[f]$ denote the homotopy class of $f$, and for any family $C$ of self-homotopy equivalences of $X$, let $[C]:=\{[f]|f\in C\}$ be
the set of homotopy classes of $C$. Then the image $\eta([C])$ is a subset of $\mathrm{Out}%
(\pi _{1}(X)).$ For a group $G$ and a subset $H$ of $G$, two elements $g,h\in H$
are conjugate if there exists an element $k\in G$ such that $g=khk^{-1}.$ Let $\bar{h}$ be the conjugacy class of $h$ in $G$, and $\mathrm{Conj}_{G}H:=\{\bar h|h\in H\}$ be the set of conjugacy classes of $H$ in $G$. We have the following:

\begin{lem}\label{finite out implies BIPH}
Let $X$ be a connected compact aspherical polyhedron. Suppose that $C$ is a family of self-homotopy equivalences of $X. $
Then we have a natural surjection
$$\Phi:~~ \mathrm{Conj}_{\mathrm{Out}(\pi_{1}(X))}\eta([C])\longrightarrow [C]_m, \quad defined~by\quad \overline{\eta([f])}\longmapsto [f]_m.$$
Moreover, if the set $\mathrm{Conj}_{\mathrm{Out}(\pi_{1}(X))}\eta([C])$ is finite,
then $[C]_m$ is also finite and hence $X$ has BIPC. In particular, when $\mathrm{Out}(\pi _{1}(X))$
has finitely many conjugacy classes, the polyhedron $X$ has the bounded index property with respect to
the set of all self-homotopy equivalences, i.e., $X$ has BIPHE.
\end{lem}

\begin{proof}
For self-homotopy equivalences $f, g\in C$,  if $\overline{\eta([f])}=\overline{\eta([g])}$, then there exists $\eta([h])=\eta([h'])^{-1}\in \out(\pi_1(X))$ for $h:X\to X$ a homotopy equivalence and  $h': X\to X$ a homotopy inverse of $h$ such that
$$\eta([g])=\eta([h])\cdot\eta([f])\cdot\eta([h])^{-1}=\eta([h])\cdot\eta([f])\cdot\eta([h'])\in \out(\pi_1(X)).$$
This implies that $g\simeq h\comp (f\comp h')$. Note that $f\simeq (f\comp h')\comp h$, so we have $[f]_m=[g]_m$. Therefore, $\Phi$ is well-defined. It is obvious from the definition that $\Phi$ is surjective, and the proof is finished by Proposition \ref{ind for muatant}.
\end{proof}

\begin{thm}
\label{main thm3} Let $M=M_1\times\cdots\times M_m$ be a product of
finitely many connected closed Riemannian manifolds, each with negative
sectional curvature everywhere, and with (not necessarily the same)
dimension $\geq 3.$ Then $M$ has BIPHE.
\end{thm}

\begin{proof}
Rips and Sela \cite{RS} building on ideas of Paulin proved that $\mathrm{Out}%
(G)$ is a finite group when $G$ is the fundamental group of a closed
Riemannian manifolds of dimension $\geq 3$ with negative
sectional curvature everywhere. Since the factors $M_{i}^{n_{i}},i=1,\ldots ,m$, are closed Riemannian manifolds, each with
negative sectional curvature everywhere, and the dimensions $n_{i}\geq 3$,
we have that $\mathrm{Out}(\pi _{1}(M))$ is also finite by Theorem \ref{aut and out of product groups}. Therefore, $M$ has BIPHE
(and hence BIPH) by Lemma \ref{finite out implies BIPH}.
\end{proof}



\section{Fixed points of cyclic homeomorphisms of products of surfaces}\label{sect 4}

In this section, we will generalize the results of alternating
homeomorphisms (see \cite[Section 3]{ZZ}) to \emph{cyclic
homeomorphisms} of products of surfaces. Let $F$ be a connected closed hyperbolic surface, and
hence, the Euler characteristics $\chi(F)<0$.

\begin{defn}
A self-homeomorphism $f$ of $F^m:=\overbrace{F\times\cdots\times F}^m$, is
called a \emph{cyclic homeomorphism}, if
\begin{equation*}
f=\tau\comp (\prod_{i=1}^m f_i): F^m\to F^m, (a_1,a_2,\ldots,a_m)\mapsto
(f_m(a_m),f_1(a_1),\ldots,f_{m-1}(a_{m-1})),
\end{equation*}
where $f_1,\ldots,f_m$ are self-homeomorphisms of $F$, and $\tau=(12\cdots
m)\in S_m$ is a $m$-cycle.
\end{defn}

Note that for a compact hyperbolic surface, every homeomorphism is isotopic
to a diffeomorphism, then by the same argument as in the proof of \cite[Lemma
3.2]{ZZ}, we have

\begin{lem}
\label{transversal} Let $f_1,\ldots, f_m$ be self-homeomorphisms of $F$, and
$f=\tau\comp (\prod_{i=1}^m f_i): F^m\to F^m$ a cyclic homeomorphism. Then $%
f_1,\ldots, f_m$ can be isotoped to diffeomorphisms $g_1,\ldots, g_m$
respectively, such that the graph of the corresponding cyclic homeomorphism $%
g=\tau\comp (\prod_{i=1}^m g_i): F^m\to F^m$ is transversal to the
diagonal in $F^m$. Moreover, $f$ is homotopic to $g$ and, for each fixed
point $(a_1, a_2,\ldots, a_m)$ of $g$, there are charts of $F^m$ at $(a_1,
a_2,\ldots, a_m)$ such that under the charts, $g$ has a local canonical
form
\begin{eqnarray}
&&(u_{11}, u_{12}, u_{21}, u_{22},\ldots, u_{m1},u_{m2})  \notag \\
&\mapsto& (g_{m1}(u_{m1}, u_{m2}), g_{m2}(u_{m1}, u_{m2}),g_{11}(u_{11},
u_{12}), g_{12}(u_{11}, u_{12}),  \notag \\
&&\ldots, g_{(m-1)1}(u_{(m-1)1}, u_{(m-1)2}),g_{(m-1)2}(u_{(m-1)1},
u_{(m-1)2}))  \notag
\end{eqnarray}
where $g_{i1}, g_{i2}$ are the components of $g_i$ under the charts.
\end{lem}

\begin{lem}
\label{index of cylic homeomorphism} If $f=\tau\comp (\prod_{i=1}^m f_i):
F^m\to F^m$ is a cyclic homeomorphism, then the natural map
\begin{equation*}
\rho: F\to F^m, \quad a_1\mapsto (a_1, f_1(a_1),\ldots, f_{m-1}\comp\cdots%
\comp f_2\comp f_1(a_1))
\end{equation*}
induces an index-preserving one-to-one corresponding between the set $%
\mathrm{Fpc}(f_{m}\comp\cdots\comp f_2\comp f_1)$ of fixed point classes of $%
f_{m}\comp\cdots\comp f_2\comp f_1$ and the set $\mathrm{Fpc}(f)$ of fixed
point classes of $f$.
\end{lem}

\begin{proof}
It is clear that
\begin{eqnarray}
\mathrm{Fix} f&=&\{(a_1,a_2,\ldots,a_m)|f_1(a_1)=a_2,
f_2(a_2)=a_3,\ldots,f_m(a_m)=a_1\}  \notag \\
&=&\{(a_1, f_1(a_1),\ldots, f_{m-1}\comp\cdots\comp f_2\comp %
f_1(a_1))|a_1\in \mathrm{Fix} (f_{m}\comp\cdots\comp f_2\comp f_1)\}.  \notag
\end{eqnarray}
Suppose that $a_1$ and $a_1^{\prime }$ are in the same fixed point class of $%
f_{m}\comp\cdots\comp f_1$, and $p: \widetilde F\to F$ is the universal
cover. Then there is a lifting $\tilde f_i$ of $f_i$ such that $%
a_1,a_1^{\prime }\in p(\mathrm{Fix} (\tilde f_m\comp\cdots\comp \tilde f_1))$%
, and there is a point $\tilde a_1\in p^{-1}(a_1)$ and a point $\tilde
a_1^{\prime -1}(a_1^{\prime })$ with $(\tilde f_m\comp\cdots\comp \tilde
f_1)(\tilde a_1) = \tilde a_1$ and $(\tilde f_m\comp\cdots\comp \tilde
f_1)(\tilde a_1^{\prime }) = \tilde a_1^{\prime }$. Hence,
\begin{eqnarray}
&&(\tau\comp(\tilde f_1\times \cdots\times\tilde f_m)) (\tilde a_1, \tilde
f_1(\tilde a_1),\ldots,\tilde f_{m-1}\comp\cdots\comp \tilde f_1(\tilde a_1))
\notag \\
&=& ((\tilde f_m\comp\cdots\comp \tilde f_1)(\tilde a_1), \tilde
f_1(a_1),\ldots,\tilde f_{m-1}\comp\cdots\comp \tilde f_1)(\tilde a_1))
\notag \\
&=&(\tilde a_1, \tilde f_1(\tilde a_1),\ldots,\tilde f_{m-1}\comp\cdots\comp %
\tilde f_1(\tilde a_1)).  \notag
\end{eqnarray}
It follows that
\begin{equation*}
(a_1, f_1(a_1),\ldots, f_{m-1}\comp\cdots\comp f_1(a_1))\in (\prod_m p)(%
\mathrm{Fix}(\tau\comp(\tilde f_1\times \cdots\times\tilde f_m))).
\end{equation*}
Similarly, we also have
\begin{equation*}
(a^{\prime }_1, f_1(a^{\prime }_1),\ldots, f_{m-1}\comp\cdots\comp %
f_1(a^{\prime }_1))\in (\prod_m p)(\mathrm{Fix}(\tau\comp(\tilde f_1\times
\cdots\times\tilde f_m))).
\end{equation*}
Since $(\tau\comp(\tilde f_1\times \cdots\times\tilde f_m))$ is a lifting of
$f$, we obtain that $(a_1, f_1(a_1),\ldots, f_{m-1}\comp\cdots\comp %
f_1(a_1)) $ and $(a^{\prime }_1, f_1(a^{\prime }_1),\ldots, f_{m-1}\comp%
\cdots\comp f_1(a^{\prime }_1))$ are in the same fixed point class of $f$.
Conversely, suppose that the two points above are in the same fixed point
class of $f$. Then there is a lifting $\tilde f_i$ of $f_i$ such that both
of them lie in $(\prod_m p)(\mathrm{Fix}(\tau\comp(\tilde f_1\times
\cdots\times\tilde f_m))) $. Hence, $a_1,a_1^{\prime }\in p(\mathrm{Fix}
(\tilde f_m\comp \cdots\comp\tilde f_1))$, we conclude that $a_1$ and $%
a_1^{\prime }$ are in the same fixed point class of $f_m\comp\cdots\comp f_2%
\comp f_1$.

Now we shall prove that as a bijective correspondence between the sets of
fixed point classes, $\rho$ is index-preserving. Since the indices of fixed
point classes are invariant under homotopies, by Lemma \ref{transversal} we
may homotope $f_i$ for $i=1,2,\ldots,m$ such that the graph of $f$ is
transversal to the diagonal, and $f$ has local canonical forms in a
neighborhood of every fixed point. Suppose that the differential $Df_i$ of $%
f_i$ at $a_i$ is $N_i=\left(
\begin{array}{cc}
\frac{\partial f_{i1}}{\partial u_{i1}} & \frac{\partial f_{i1}}{\partial
u_{i2}}\\
\frac{\partial f_{i2}}{\partial u_{i1}} & \frac{\partial f_{i2}}{\partial
u_{i2}}
\end{array}
\right).$ Then the differential $Df$ of $f$ at $(a_1,a_2,\ldots,a_m)$ is
\begin{equation*}
N=\left(%
\begin{array}{ccccc}
0 & 0 & \cdots & 0 & N_m \\
N_1 & 0 & \cdots & 0 & 0 \\
0 & N_2 & \cdots & 0 & 0 \\
\vdots & \vdots & \ddots & \vdots & \vdots \\
0 & 0 & \cdots & N_{m-1} & 0%
\end{array}
\right).
\end{equation*}
Therefore, the index of $f_m\comp\cdots\comp f_2\comp f_1$ at the fixed
point $a_1$ is
\begin{equation*}
\mathrm{ind}(f_m\comp\cdots\comp f_2\comp f_1, a_1) = sgn\det(I_2-N_m\cdots
N_2N_1),
\end{equation*}
and the index of $f$ at the fixed point $(a_1,a_2,\ldots,a_m)$ is
\begin{eqnarray}
\mathrm{ind}(f, (a_1,a_2,\ldots,a_m))&=&sgn \det(I_{2m}-N)  \notag \\
&=&sgn\det(I_2-N_{m-1}\cdots N_1N_m)  \notag \\
&=& sgn\det(I_2-N_m\cdots N_2N_1) ,  \notag
\end{eqnarray}
where $I_k$ is the identity matrix of order $k$. Therefore,
\begin{equation*}
\mathrm{ind}(f_m\comp\cdots\comp f_2\comp f_1, a_1)=\mathrm{ind}(f,
(a_1,a_2,\ldots,a_m)),
\end{equation*}
and the proof is finished.
\end{proof}

As a corollary, we have

\begin{cor}
\label{alter for N&L}
\begin{equation*}
N(f)=N(f_m\comp\cdots\comp f_2\comp f_1), \quad L(f)=L(f_m\comp\cdots\comp %
f_2\comp f_1).
\end{equation*}
\end{cor}

Directly following from Lemma \ref{index of cylic homeomorphism}, Corollary %
\ref{alter for N&L} and \cite[Theorem 4.1]{JG}, we have

\begin{prop}
\label{bounds for cyclic homeomorphism} If $f: F^m\to F^m$ is a cyclic
homeomorphism, then

$\mathrm{(A)}$ For every fixed point class $\mathbf{F}$ of $f$, we have
\begin{equation*}
2\chi(F)-1\leq \mathrm{ind}(f,\mathbf{F})\leq 1.
\end{equation*}
Moreover, almost every fixed point class $\mathbf{F}$ of $f$ has index $\geq
-1$, in the sense that
\begin{equation*}
\sum_{\mathrm{ind}(f,\mathbf{F})<-1}\{\mathrm{ind}(f,\mathbf{F})+1\}\geq
2\chi(F),
\end{equation*}
where the sum is taken over all fixed point classes $\mathbf{F}$ with $%
\mathrm{ind}(f,\mathbf{F})<-1$;

$\mathrm{(B)}$ Let $L(f)$ and $N(f)$ be the Lefschetz number and the Nielsen
number of $f$ respectively. Then
\begin{equation*}
|L(f)-\chi(F)|\leq N(f)-\chi(F).
\end{equation*}
\end{prop}


\section{Fixed points of product maps and proofs of Theorem \ref{main thm1} and \ref{main thm2}}\label{sect 5}

\label{sect. of borel conjecture}

To prove Theorem \ref{main thm1} and Theorem \ref{main thm2}, we need some facts about fixed points of product maps.

\subsection{Fixed points of product maps}

Let $X_1,\ldots,X_n$ be connected compact polyhedra.

\begin{defn}
A self-map $f: X_1\times\cdots\times X_n\to X_1\times\cdots\times X_n$ is
called a \emph{product map}, if
\begin{equation*}
f=f_1\times\cdots\times f_n: X_1\times\cdots\times X_n\to
X_1\times\cdots\times X_n,\quad (a_1,\ldots,a_n)\mapsto
(f_1(a_1),\ldots,f_n(a_n)),
\end{equation*}
where $f_i$ is a self-map of $X_i, i=1,\ldots,n$.
\end{defn}

By a proof analogous to that of \cite[Lemma 2.2]{ZZ}, we have the following lemma
about the fixed point classes of product maps.

\begin{lem}
\label{product of index} If $f: X_1\times\cdots\times X_n\to
X_1\times\cdots\times X_n$ is a product map, then $\mathrm{Fix} f=%
\mathrm{Fix} f_1\times\cdots\times \mathrm{Fix} f_n,$ and each fixed point
class $\mathbf{F}\in \mathrm{Fpc}(f)$ splits into a product of some fixed
point classes of $f_i$, i.e.,
\begin{equation*}
\mathbf{F}=\mathbf{F}_1\times\cdots\times \mathbf{F}_n, \quad\mathrm{ind}(f,%
\mathbf{F})=\mathrm{ind}(f_1,\mathbf{F}_1)\cdots\mathrm{ind}(f_n, \mathbf{F}%
_n),
\end{equation*}
where $\mathbf{F}_i\in \mathrm{Fpc} (f_i)$ is a fixed point class of $f_i$
for $i=1,\ldots,n$. Moreover,
\begin{equation*}
L(f)=L(f_1)\cdots L(f_n),\quad N(f)=N(f_1)\cdots N(f_n).
\end{equation*}
\end{lem}

\subsection{Proofs of Theorem \ref{main thm1} and Theorem \ref{main thm2}}

Now we can give the proofs of Theorem \ref{main thm1} and \ref{main thm2}. Since the index of fixed points is homotopy invariant, we omit the base
points of fundamental groups in the following.

\begin{proof}[Proof of Theorem \protect\ref{main thm1}]
Let $X=X_1\times\cdots\times X_n$, and $f:X\to X$ a homotopy equivalence.
Then $f$ induces an automorphism $f_{\pi}:\pi _{1}(X)\rightarrow \pi_{1}(X)$%
. Note that $\pi_1(X)$ is isomorphic to the direct product $%
\prod_{i=1}^{n}\pi_1(X_i)$. By Proposition \ref{aut of prouct group}, the
condition (1) in Theorem \ref{main thm1} implies $f_{\pi}=\phi_{1}\times%
\cdots\times \phi_n$ with $\phi_{i}$ an automorphism of $\pi _{1}(X_i),
i=1,\dots, n$. Note that $X_i$ is a compact aspherical polyhedron, so $\phi_{i}$
can be induced by a homotopy equivalence $f_{i}:X_i\to X_i, i=1,\dots, n$.
Since the product $X$ is also an compact aspherical polyhedron, $f$ is
homotopic to the product map $f_{1}\times\cdots\times f_n$ which is
also a homotopy equivalence. Recall that $X_i$ has BIPHE, then the index $%
\mathrm{ind}(f_i,\mathbf{F}_i)$ of any fixed point class $\mathbf{F}_i$ of $%
f_i$ has a finite bound $\mathcal{B}_{X_i}$ depending only on $X_i$. By the
product formula of index in Lemma \ref{product of index}, we have the index $%
|\mathrm{ind}(f,\mathbf{F})|<\mathcal{B}_X:=\prod_{i=1}^n\mathcal{B}_{X_i}$
for every fixed point class $\mathbf{F}$ of $f$. Therefore, $X$ has BIPHE.
\end{proof}

\begin{proof}[Proof of Theorem \protect\ref{main thm2}]
Let $M_1,\ldots, M_n$ be connected closed Riemannian manifolds, each with
negative sectional curvature everywhere, and $M=M_1\times\cdots\times M_n$.
Collect together the coordinates corresponding to homotopy equivalent $M_i$'s and
present it in the form
\begin{equation*}
M=\prod^s_{i=1}M_i^{\,n_i}\times \prod^m_{i=s+1}M_i^{\,n_i},
\end{equation*}
where $M_1,\ldots, M_s$, $0\leq s\leq n$, are hyperbolic surfaces, $%
M_{s+1},\ldots, M_m$ have dimensions $\geq 3$, $n_i\geq 1$ (recall that $n_i$ is not the dimension but the number of copies of $M_i$),
$n_1+\cdots+n_m=n $ and $M_i\not\simeq M_j$ for $1\leq i\neq j\leq m$. Then by
Example \ref{centerless of negatived mfd}, the fundamental group $G_i:=\pi_1(M_i)$ is centerless and
indecomposable for $i=1,2,\ldots n$. Therefore,
\begin{equation*}
\pi_1(M)=\prod^s_{i=1}G_i^{\,n_i}\times \prod^m_{i=s+1}G_i^{\,n_i},
\end{equation*}
where $G_i\ncong G_j$ for $1\leq i\neq j\leq m$.

For any homotopy equivalence $f:M\to M$, $f$ induces an automorphism $%
f_{\pi} $ of $\pi_1(M)$, then by Proposition \ref{aut of prouct group},
there exist automorphisms $\phi_{i,j}\in \mathrm{Aut}(G_i)$ and permutations
$\sigma_i \in S_{n_i}$, such that
\begin{equation*}
f_{\pi}= \prod_{i=1}^s (\sigma_i \comp \prod_{j=1}^{n_i} \phi_{i,j})\times
\prod_{i=s+1}^m (\sigma_i \comp \prod_{j=1}^{n_i} \phi_{i,j}):=\phi\times
\psi,
\end{equation*}
where $\phi=\prod_{i=1}^s (\sigma_i \comp \prod_{j=1}^{n_i} \phi_{i,j})$ and
$\psi=\prod_{i=s+1}^m (\sigma_i \comp \prod_{j=1}^{n_i} \phi_{i,j})$. Recall
that $G_i$ for $i>s$ is the fundamental group of a Riemannian manifold with
dimensions $\geq 3$, then $\psi$ can be induced by a homotopy equivalence $%
h: \prod^m_{i=s+1}M_i^{\,n_i}\to \prod^m_{i=s+1}M_i^{\,n_i}.$ On the other
hand, $G_i$ ($i\leq s$) is the fundamental group of a closed hyperbolic
surface, so the automorphism $\phi_{i,j}\in \mathrm{Aut}(G_i)$ can be
induced by a homeomorphism $f_{i,j}$, and hence $\phi$ is induced by the
homeomorphism
\begin{equation*}
g=\prod_{i=1}^s(\sigma_i \comp \prod_{j=1}^{n_i}
f_{i,j}):\prod^s_{i=1}M_i^{\,n_i}\to \prod^s_{i=1}M_i^{\,n_i}.
\end{equation*}
That is $\phi=g_{\pi}$ and thus $f_{\pi}=g_{\pi}\times h_{\pi}=(g\times
h)_{\pi} $. Since $M$ is also aspherical, $f$ is homotopic to the
product map $g\times h$. By Theorem \ref{main thm3}, for every
fixed point class $\mathbf{F}$ of $h$, we have $|\mathrm{ind}(h,\mathbf{F})|<%
\mathcal{B}_M$ for some finite bound $\mathcal{B}_M$ depending only on $M$.

To complete the proof, by Lemma \ref{product of index}, it suffices to show
that $|\mathrm{ind}(g,\mathbf{F}^{\prime })|<\mathcal{B}^{\prime }_M$ for
some finite bound $\mathcal{B}^{\prime }_M$ depending only on $M$, for every
fixed point class $\mathbf{F'}$ of $g$. Since every permutation $\sigma_i\in
S_{n_i}$ is a product of disjoint cycles, and $M_1,\ldots, M_s$ are
hyperbolic surfaces, we can rewritten
\begin{equation*}
g=\prod_{k} g_k:\prod^s_{i=1}M_i^{\,n_i}\to \prod^s_{i=1}M_i^{\,n_i}
\end{equation*}
as a product of finitely many cyclic homeomorphisms $g_k$ of products of
hyperbolic surfaces. Then by Proposition \ref{bounds for cyclic
homeomorphism}, we can choose $\B^{\prime
}_M=\prod^s_{i=1}|2\chi(M_i)-1|^{\,n_i}.$
\end{proof}





\begin{thebibliography}{JWZ}
\bibitem[BH]{bh} M. Bridson and A. Haefliger, \textit{Metric Spaces of
Non-Positive Curvature}, Grund. Math. Wiss. 319, Springer-Verlag,
Berlin-Heidelberg-New York, 1999.

\bibitem[J1]{fp1} B. Jiang, \emph{Lectures on Nielsen Fixed Point Theory},
Contemporary Mathematics vol.~14, American Mathematical Society, Providence
(1983).

\bibitem[J2]{fp2} B. Jiang, \emph{Bounds for fixed points on surfaces},
Math. Ann. 311 (1998), 467--479.

\bibitem[JG]{JG} B. Jiang and J. Guo, \emph{Fixed points of surface
diffeomorphisms}, Pac. J. Math. 160 (1) (1993), 67--89.

\bibitem[JW]{JW} B. Jiang and S. Wang, \emph{Lefschetz numbers and Nielsen
numbers for homeomorphisms on aspherical manifolds}, Topology Hawaii, 1990,
World Sci. Publ., River Edge, NJ, 1992, 119--136.

\bibitem[JWZ]{JWZ} B. Jiang, S.D. Wang and Q. Zhang, \emph{Bounds for fixed
points and fixed subgroups on surfaces and graphs}, Algebr. Geom. Topol. 11
(2011), 2297--2318.

\bibitem[K1]{K1} M. Kelly, \emph{A bound for the fixed-point index for surface
mappings}, Ergodic Theory Dynam. Systems 17 (1997), 1393--1408.

\bibitem[K2]{K2} M. Kelly, \emph{Bounds on the fixed point indices for self-maps
of certain simplicial complexes}, Topology Appl. 108 (2000), 179--196.

\bibitem[Mc]{Mc} C. McCord,\emph{Estimating Nielsen numbers on
infrasolvmanifolds}, Pac. J. Math. 154 (1992), 345--368.

\bibitem[RS]{RS} E. Rips and Z. Sela, \emph{Structure and rigidity in
hyperbolic groups. I}, Geom. Funct. Anal. 4 (1994), no. 3, 337--371.

\bibitem[Z1]{Z} Q. Zhang, \emph{Bounds for fixed points on Seifert manifolds}%
, Topology Appl. 159 (15) (2012), 3263--3273.

\bibitem[Z2]{Z2} Q. Zhang, \emph{Bounds for fixed points on hyperbolic
3-manifolds}, Topology Appl. 164 (2014), 182--189.

\bibitem[Z3]{Z3} Q. Zhang, \emph{Bounds for fixed points on hyperbolic
manifolds}, Topology Appl. 185-186 (2015), 80--87.

\bibitem[ZVW]{ZVW} Q. Zhang, E. Ventura and J. Wu, \emph{Fixed subgroups are
compressed in surface groups}, Internat. J. Algebra Comput. 25 (5) (2015),
865--887.

\bibitem[ZZ]{ZZ} Q. Zhang and X. Zhao, \emph{Bounds for fixed points on products
of hyperbolic surfaces}, J. Fixed Point Theory Appl. 21: 6 (2019), 1--11.
\end{thebibliography}
\end{document}